\documentclass[11pt]{article}
\usepackage{amssymb, amsmath}
\usepackage{amsthm}
\usepackage{caption}
\usepackage{subcaption}
\usepackage[margin=1in]{geometry}
\usepackage{enumerate}
\allowdisplaybreaks[1]
\numberwithin{equation}{section}

\newtheorem{example}{Example}[section]
\newtheorem{thm}{Theorem}[section]
\newtheorem{cor}{Corollary}[section]

\newtheorem{pro}{Proposition}[section]

\usepackage{graphicx}

\begin{document}
\markboth{R. Rajkumar and P. Devi}{}
\title{Coprime graph of subgroups of a group}

\author{R. Rajkumar\footnote{e-mail: {\tt rrajmaths@yahoo.co.in}},\ \ \
P. Devi\footnote{e-mail: {\tt pdevigri@gmail.com}}\\
{\footnotesize Department of Mathematics, The Gandhigram Rural Institute -- Deemed University,}\\ \footnotesize{Gandhigram -- 624 302, Tamil Nadu, India.}\\[3mm]
}
\date{}
\maketitle

\maketitle

\begin{abstract}
Let $G$ be a group. We define the \emph{coprime graph of subgroups} of $G$, denoted by $\mathcal P(G)$, is a graph whose vertex set is the set of all proper subgroups of $G$, and two distinct vertices are adjacent if and only if the order of the corresponding subgroups are coprime. In this paper, we study some connections between algebraic properties
of a group and graph theoretic properties of its coprime graph.
\end{abstract}

\section{Introduction}
Graph theory provide tools to study the algebraic properties of algebraic structures. In particular, there are several graphs associated with groups to study some specific properties of groups, for instance, intersection graph of subgroups of groups, prime graph of groups, non-commuting graphs of groups and permutability graph of subgroups of groups (See \cite{abd}, \cite{hadi}, \cite{mogha}, \cite{raj}  and the references therein). In \cite{sat}, Sattanathan and Kala defined  the order prime graph of a group $G$, which is a graph having the set of all elements of $G$ as its vertices, and two distinct vertices are adjacent if and only if the orders of the corresponding subgroups are coprime. They have studied some properties of this graph. This graph was further investigated by Xuanlong Ma et al \cite{ma} and  Hamid Reza Dorbidi \cite{hamid}. They called the order prime graph of a group as the coprime graph of a group.

The relation of coprimeness of the orders of the subgroups of a group plays a significant role in the determination of the structural properties of that group. 
Also the relation of  coprimeness is not transitive on the set of all proper subgroups of a group.   
In this paper, we define the \emph{coprime graph of subgroups} of $G$, denoted by $\mathcal P(G)$. It is a graph having all the proper subgroups of $G$ as its vertices, and two distinct vertices $H$ and $K$ are adjacent if and only if $|H|$ and $|K|$ are coprime.


Now we recall some basic definitions and notations of graph theory. We use the standard terminology of graphs (e.g., see \cite{harary}).
 Let $G$ be a simple graph. $G$ is said to be $k$-partite if the vertex set of $G$ can be partitioned to $k$ sets such that no two vertices in same partitions are adjacent. A complete $k$-partite graph, denoted by $K_{n_1, n_2, \cdots, n_k}$, is a $k$-partite graph having partition sizes $n_1, n_2, \cdots, n_k $ such that  every vertex in each partition is adjacent with all the vertices in the remaining partitions. In particular, $K_{1,n}$ is called a \emph{star}.    A graph whose edge set is empty is called a \emph{null} graph or \emph{totally disconnected} graph.  $K_n$ denotes the complete graph on $n$ vertices.   $P_n$ and $C_n$ respectively denotes the path and cycle with $n$ edges. We denote  the degree of
a vertex $v$ in $G$ by $deg_G(v)$.
 A graph is said to be \emph{connected} if any two vertices of it can be joined by a path. The	\emph{diameter} of a connected graph is the maximum of the length of the shortest path between any pair of vertices.
 A \emph{tree} is a connected graph with out cycles.  $G$ is said to be  $H$-\textit{free} if $G$ has no subgraph isomorphic to $H$. The girth of $G$, denoted by $girth(G)$, is the length of its shortest cycle, if it exist; other wise $girth(G) = \infty$. An \emph{independent set} of $G$ is a subset of $V(G)$ having no
two vertices are adjacent. The \emph{independence number} of $G$, denoted by $\alpha(G)$, is the
cardinality of the largest independent set. A \emph{clique} of  $G$ is a complete subgraph of $G$. The clique number $\omega(G)$ of $G$ is the cardinality of a largest clique in $G$.
The \emph{chromatic number} $\chi(G)$ of $G$ is the smallest number of colours needed to colour the vertices
of $G$ such that no two adjacent vertices gets the same colour. G is said to be \emph{weakly
$\chi$-perfect} if $\omega(G) = \chi(G)$.  A graph is said to be\textit{ planar}, if it can be drawn in the plane, so that no two 
lines intersect except at the vertices; otherwise it is said to be \textit{nonplanar}.  A graph is called \textit{unicyclic}, if it contains 
exactly one cycle.

For any integer $n >1$, $\pi(n)$ denotes the set of all prime divisors of $n$. If $G$ is a finite group, then $\pi(|G|)$ is denoted by $\pi(G)$. The number of Sylow p-subgroups of a group G is denoted by $n_p(G)$. We denote the order of an element $a\in\mathbb Z_n$ by $\mbox{ord}_n(a)$. Moreover, through out this paper, $p$, $q$, $r$, $s$ denotes the distinct primes.

Since the only groups having no proper subgroups are the trivial group, and the
groups of prime order, it follows that, we can define $\mathcal P(G)$ only when the group $G$ is
neither the trivial group nor the group of prime order. So, unless
otherwise mentioned, throughout this paper we consider only groups other than the
trivial group, and the groups of prime order.

We use only elementary methods. In Section 2,  we classify all the finite groups whose coprime graph of subgroups are one of  totally disconnected, bipartite, connected, complete, complete bipartite, tree,  star or path, and show that  the coprime graph of subgroups of a finite group can not be a cycle.  For a finite group $G$, we obtain the independence number, clique number, chromatic number, diameter, girth of $\mathcal P(G)$, and show that  $\mathcal P(G)$  is weakly $\chi$-perfect. Moreover, we obtain the degrees of vertices of $\mathcal P(\mathbb Z_n)$, and we show that every simple graph is an induced subgraph of $\mathcal P(\mathbb Z_n)$, for some $n$.
 
In Section 3, we classify all the finite groups whose coprime graph of subgroups of groups are one of
planar,  $K_{2,3}$-free, $K_{1,4}$-free, $K_{1,3}$-free, $K_{1,2}$-free, unicyclic.

\section{Some results on coprime graph of subgroups of groups}

\begin{thm}\label{order prime graph of subgroups 14}
Let $G_1$ and $G_2$ be two groups. If $G_1\cong G_2$, then $\mathcal P(G_1)\cong \mathcal P(G_2)$.
\end{thm}
\begin{proof}
Let $f:G_1\rightarrow G_2$ be a group isomorphism. Define a map $\psi:V(\mathcal P(G_1))\rightarrow V(\mathcal P(G_2))$ by $\psi(H)=f(H)$, for every $H\in V(\Gamma(G_1))$.
Since a group isomorphism preserves the order of subgroups, so it follows that $\psi$ is  a graph isomorphism.
\end{proof}

\noindent\textbf{Remark:} The converse of the above Theorem~\ref{order prime graph of subgroups 14}, is not true, for if $G_1\cong \mathbb Z_{p^5}$ and $G_2\cong Q_8$, then the number of  proper subgroups of $G_1$ is four and their orders are $p$, $p^2$, $p^3$, $p^4$; the number of  proper subgroups of $G_2$ is four and their orders are 4, 4, 4, 2. Here $\mathcal P(G_1)\cong \overline{K}_4\cong \mathcal P(G_2)$, but $G_1\ncong G_2$.


\begin{thm}\label{order prime graph of subgroups of groups 4}
Let $G$ be a group of order $p_1^{\alpha_1}p_2^{\alpha_2}\ldots p_k^{\alpha_k}$, where $p_i$'s are distinct primes,  $\alpha_i\geq 1$.
Then
\begin{enumerate}[\normalfont (1)]
\item $\mathcal P(G)$ is $k$-partite;
\item $\alpha(\mathcal P(G))= \displaystyle \max_{\substack{i}}~ |\mathcal B_i|$, where for each $i=1$, 2, $\ldots$, $k$, $\mathcal B_i$ is the set of all proper subgroups of $G$ whose order is divisible by $p_i$;
\item $\omega(\mathcal P(G))=k=\chi(\mathcal P(G))$;
\end{enumerate}
In particular, $\mathcal P(G)$ is weakly $\chi$-perfect.
\end{thm}
\begin{proof}
Let $\mathcal A_1$ be the set of all proper subgroups of $G$ whose order is divisible by $p_1$. For each $i \in \{2$, 3, $\ldots$, $k\}$, let $\displaystyle \mathcal A_i =\{H~|~H~\mbox{is a proper subgroup of G such that}~ p_i~\mbox{divides}~|H|\}-\bigcup_{j=1}^{i-1} \mathcal A_j$. Then clearly the collection $\{ \mathcal A_i\}_{i=1}^k$ forms a partition of the vertex set of $\mathcal P(G)$. Also no two vertices in a same partition are adjacent in $\mathcal P(G)$. Moreover, $k$ is the minimal number such that a $k$-partition of the vertex set of $\mathcal P(G)$ is having this property, since $\pi(G)=k$.  It follows that $\mathcal P(G)$ is $k$-partite.

Now for each $i=1$, 2, $\ldots$, $k$, let $\mathcal B_i$ be the set of all proper subgroups of $G$ whose order is divisible by $p_i$. Clearly each $\mathcal B_i$ is a maximal independent set of $\mathcal P(G)$. Thus  $\alpha(\mathcal P(G))= \displaystyle \max_{\substack{i}}~ |\mathcal B_i|$.
For each $i=1, 2, \ldots, k$, $G$ has a subgroup of order $p_i$. Then the set having one subgroup from each of these orders forms a clique in $\mathcal P(G)$. Since $\mathcal P(G)$ is $k$-partite, it follows that $\omega(\mathcal P(G))=k$.
Obviously, $\chi(\mathcal P(G))=k$. Weakly $\chi$-perfectness of $\mathcal P(G)$ follows from the definition.
\end{proof}

The next result is an immediate consequence of Theorem~\ref{order prime graph of subgroups of groups 4} (1).

\begin{cor}\label{c1}
Let $G$ be a group with $\pi(G)=k$.
Then
\begin{enumerate}[\normalfont (1)]
\item $\mathcal P(G)$ is totally disconnected if and only if $k=1$;
\item $\mathcal P(G)$ is bipartite if and only if $k =1, 2$.
\end{enumerate}
\end{cor}


\begin{thm}\label{order prime graph of subgroups of groups 6}
Let $G$ be a finite group. Then
\begin{enumerate}[\normalfont (1)]
\item $\mathcal P(G)$ is complete bipartite if and only if $G$ is isomorphic to one of $\mathbb Z_{pq}$,
$\mathbb Z_q\rtimes \mathbb Z_p$, $(\mathbb Z_p\times \mathbb Z_p)\rtimes \mathbb Z_q$ or $A_4$;
\item The following are equivalent:
\begin{enumerate}[\normalfont (a)]
\item $G \cong \mathbb Z_{pq}$ or $\mathbb Z_q\rtimes \mathbb Z_p$;
\item $\mathcal P(G)$ is a tree;
\item $\mathcal P(G)$ is a star.
\end{enumerate}

\item The following are equivalent:
\begin{enumerate}[\normalfont (a)]
\item $G \cong \mathbb Z_{pq}$;
\item $\mathcal P(G)$ is complete;
\item $\mathcal P(G)$ is a path.
\end{enumerate}
\end{enumerate}
\end{thm}
\begin{proof}
In view of part (2) of Corollary~\ref{c1}, to prove parts (1), (2), and $(a) \Leftrightarrow (c)$ of (3) of this theorem, it is enough to consider the groups of order $p^\alpha$ and $p^\alpha q^\beta$.

If $|G|=p^\alpha$, then by Corollary~\ref{c1}(1), $\mathcal P(G)$ is totally disconnected and so it is neither complete bipartite nor a tree.

Let $|G|=pq$ with $p < q$. Then $G\cong \mathbb Z_{pq}$ or $\mathbb Z_q\rtimes \mathbb Z_p$. If $G\cong \mathbb Z_{pq}$, then
$\mathcal P(G)\cong K_2$
and so $\mathcal P(G)$ is a path. If $G\cong \mathbb Z_q\rtimes \mathbb Z_p$, then $G$ has an unique subgroup of order $q$, and $q$ subgroups of order $p$; also these are the only proper subgroups of $G$. It follows that
$\mathcal P(G)\cong K_{1,q}$
and so $\mathcal P(G)$ is complete bipartite, star; but not a path.

Let $|G|=p^2q$. Suppose $G$ is abelian, then $G$ has a subgroup of order $pq$ and so $\mathcal P(G)$ is disconnected. It follows that $\mathcal P(G)$ is neither  complete bipartite nor a tree. Now assume that $G$ is non-abelian. Here we use the classification of groups of order $p^2q$ given in \cite[p.~76-80]{burn}.

\noindent\textbf{Case 1:} $p<q$:

\noindent\textbf{Case 1a:} $p \nmid (q-1)$. By Sylow's Theorem, it is easy to see that there is no non-abelian group in this case.

\noindent\textbf{Case 1b:} $p \mid (q-1)$ but $p^2 \nmid (q-1)$. In this case, there are two non-abelian groups.

The first group is $G_1:= \mathbb Z_q \rtimes \mathbb Z_{p^2} = \langle a, b | a^q= b^{p^2}=1, bab^{-1}=a^i, ord_q(i)=p \rangle$. Here $G_1$ has an unique subgroup of order $q$; unique subgroup of order $pq$; $q$ subgroups of order $p^2$ and unique subgroups of order $p$; also these are the only proper subgroups of $G_1$. Therefore, $\mathcal P(G_1)\cong K_{1,q+1}\cup K_1$. So $\mathcal P(G_1)$ is disconnected. Hence $\mathcal P(G_1)$ is neither complete bipartite nor a tree.

The second group is $G_2:= \langle a, b, c| a^q= b^p= c^p=1, bab^{-1}=a^i, ac=ca, bc=cb, ord_q(i)=p \rangle$. Here $G_2$ has a subgroup $\langle a, c\rangle$ of order $pq$ and so $\mathcal P(G_2)$ is disconnected. Hence $\mathcal P(G_2)$ is neither complete bipartite nor a tree.

\noindent\textbf{Case 1c:} $p^2 | (q-1)$. In this case, we have both groups $G_1$ and $G_2$ from Case 1b together with the group $G_3:= \mathbb Z_q
\rtimes_2 \mathbb Z_{p^2} = \langle a, b | a^q= b^{p^2}=1, bab^{-1}=a^i, ord_q(i)=p^2 \rangle$. Here $G_3$ has unique subgroup of order $q$; unique subgroup of order $pq$; $q$ subgroups of order $p^2$, and $q$ subgroups of order $p$. Also these are the only proper subgroups of $G_3$. Therefore, $\mathcal P(G_3)\cong (K_1+K_{2q})\cup K_1$ and so $\mathcal P(G_3)$ is disconnected. Hence $\mathcal P(G_3)$ is neither complete bipartite nor a tree.

\noindent\textbf{Case 2:} $p > q$

\noindent\textbf{Case 2a:} $q \nmid (p^2 -1)$. Then there is no  non-abelian subgroups.

\noindent\textbf{Case 2b:} $q | (p-1)$. In this case, we have two groups.

The first one is $G_4:=\langle a,b | a^{p^2}= b^q=1, bab^{-1}, ord_{p^2}(i)=q \rangle$. Here $G_4$ has a unique subgroup of order $p^2$; unique subgroup of order $p$; $p$ subgroups of order $pq$; $p^2$ subgroups of order $q$. Also these are the only proper subgroups of $G_4$. Therefore, $\mathcal P(G_4)\cong (K_2+K_{p^2})\cup \overline{K}_{p}$ and so $\mathcal P(G_4)$ is disconnected. Hence $\mathcal P(G_4)$ is neither complete bipartite nor a tree.

Next we have the family of groups $\langle a, b, c | a^p=b^p=c^q=1, cac^{-1}=a^i, cbc^{-1}=b^{i^t}, ab=ba, ord_p(i)=q \rangle$. There are $(q+3)/2$
 isomorphism types in this family (one for $t=0$ and one for each pair $\{ x, x^{-1} \}$ in ${F_p}^{\times}$. We will refer to all of these groups as $G_{5(t)}$
of order $p^2q$. Here $G_{5(t)}$ has a
subgroup $\langle a, c\rangle$ of order $pq$ and so $\mathcal P(G_{5(t))}$ is disconnected. Hence $\mathcal P(G_{5(t)})$ is neither complete bipartite nor a tree.

\noindent\textbf{Case 2c:} $q| (p+1)$. In this case, we have only one subgroup of order $p^2q$, given by $G_6:= (\mathbb Z_p \times \mathbb Z_p) \rtimes \mathbb Z_q
= \langle a, b,c| a^p=b^p=c^q=1, ab=ba, cac^{-1}=a^ib^j, cbc^{-1}=a^kb^l\rangle$, where $\bigl(\begin{smallmatrix}
  i & j\\ k & l
\end{smallmatrix} \bigr)$ has order $q$ in $GL_2(p)$. Here $G_6$ does not have a
subgroup of order $pq$. But $G_6$ has an unique subgroup of order $p^2$; $p+1$ subgroups of order $p$; $p^2$ subgroups of order $q$, also these are the only proper subgroups of $G_6$. Hence $\mathcal P(G_6)\cong K_{p+2}+K_{p^2}$, which is complete bipartite; but which is not a tree.

Note that if $(p, q)= (2, 3)$, the Cases 1 and 2 are not mutually exclusive. Up to isomorphism, there are three non-abelian groups of order 12:
$\mathbb Z_3 \rtimes \mathbb Z_4$, $D_{12}$ and $A_4$. In Case 1b, we already dealt with the group $\mathbb Z_3\rtimes \mathbb Z_4$. It is easy to see that $\mathcal P(D_{12})\cong K_{1,10}\cup \overline{K}_3$ and $\mathcal P(A_4)\cong K_{4,4}$. These three graphs are not trees and the only complete bipartite graph is $\mathcal P(A_4)$.


If $|G|=p^\alpha q$, $\alpha\geq 3$, then
let $a$, $b$ be elements in $G$ of order $p$, $q$ respectively. Here $\langle a, b\rangle$ is a proper subgroup of $G$ whose order is divisible by $p$ and $q$. Therefore,
$\mathcal P(G)$ is disconnected. Hence $\mathcal P(G)$ is neither complete bipartite nor a tree.

If $|G|=p^\alpha q^\beta$, $\alpha$, $\beta\geq 2$, then $G$ has a subgroup with prime index, since $G$ is solvable and so $\mathcal P(G)$ is disconnected. Hence $\mathcal P(G)$ is neither complete bipartite nor a tree.

Combining all the above cases together, the proof of parts (1), (2), and $(a) \Leftrightarrow (c)$ of (3) of this theorem follows.

Now, we prove $(a) \Leftrightarrow (b)$ of part (3): Clearly $(a) \Rightarrow (b)$. So assume that $\mathcal P(G)$ is complete. Then by Theorem \ref{order prime graph of subgroups of groups 4}, each partition $\mathcal A_i$, $i= 1, 2, \ldots , k$ of $\mathcal P(G)$ must contain exactly one subgroup of distinct prime order and so these subgroups are normal in $G$. If $k >3$, then $G$ must contain a subgroup whose order is a product of $k$ distinct primes, so this subgroup is an isolated vertex in  $\mathcal P(G)$, which is not possible. Hence $ k = 2$ and so by part (1) of this theorem, it turns out that $G \cong \mathbb Z_{pq}$. This completes the proof.
\end{proof}

\begin{thm}\label{order prime graph of subgroups of groups 8}
If $G$ is a finite group, then $\mathcal P(G)\ncong C_n$, for $n\geq 3$.
\end{thm}
\begin{proof}
Suppose $\mathcal P(G)$ is the cycle $H_1 - H_2 - \cdots - H_n - H_1$ of length $n$. Since $(|H_1|, |H_2|)=1 = (|H_2|, |H_3|)$, so without loss of generality, we assume that,
$|H_1|=p$, $|H_2|=q$ and  $|H_3|=r$ or $p^\alpha$. If $|H_3|=r$, then $H_1$, $H_2$, $H_3$ are adjacent and so $\mathcal P(G)$ is complete, which is not possible, by Theorem~\ref{order prime graph of subgroups of groups 6}(3). If $|H_3|=p^\alpha$, then $(|H_3|, |H_4|)=1$ implies that
$|H_4|=q^\beta$ or $r$. If $|H_4|=r$, then $H_1$, $H_2$, $H_4$ are adjacent, which is not possible. So we have $|H_4|=q^\beta$. Then $(|H_1|, |H_4|)=1$ and so $H_1$ and $H_4$ are adjacent in $\mathcal P(G)$. It follows that $n=4$ and $|G|=p^\alpha q^\beta$, $\alpha$, $\beta\geq 1$. Now we check the existence of such a group. If $\alpha+\beta\geq 4$, then $G$ has at least five proper subgroups, which is not possible. If $\alpha+\beta\leq 3$, then $|G|=p^2q$ or $pq$. In this case, we have shown in the proof of Theorem~\ref{order prime graph of subgroups of groups 6} that $\mathcal P(G)$ can not be a cycle. This completes the proof.
\end{proof}

\begin{thm}\label{order prime graph of subgroups of groups 3}
Let $G$ be a finite group. Then $\mathcal P(G)$ is connected if and only if $G$ does not have a proper subgroup $H$ with $\pi(H)=\pi(G)$. In this case, $diam (\mathcal P(G)) \in \{1,2,3\}$.
\end{thm}
\begin{proof}
Suppose $G$ has a subgroup, say $H$ with $\pi(H)=\pi(G)$. Then $|H|$ is not relatively prime to any other subgroups of $G$.
Therefore, $\mathcal P(G)$ is disconnected.
Conversely, assume that $G$ does not have a subgroup $H$ with $\pi(H)=\pi(G)$. Suppose $\mathcal P(G)$ is complete, then $\mathcal P(G)$ is connected and $diam(\mathcal P(G))=1$. Now assume that $\mathcal P(G)$ is not complete. Let $H$ and $K$ be two non-adjacent vertices in $\mathcal P(G)$.
Then by assumption, $\pi(H)\neq\pi(G)$ and $\pi(K)\neq\pi(G)$, and so there exist $p_i$, $p_j\in \pi(G)$ such that $p_i\notin \pi(H)$ and $p_j\notin \pi(K)$. If $p_i=p_j$, then there is a path $H-H_1-K$, where $H_1$ is a subgroup of $G$ of order $p_i$. If $p_i\neq p_j$, then there is a path $H-H_1-H_2-K$, where $H_1$, $H_2$ are subgroups of $G$ of orders $p_i$, $p_j$ respectively.  It follows that $\mathcal P(G)$ is connected and $diam(\mathcal P(G))\leq 3$. Note that $diam(\mathcal P(Z_{pq}))=1$, $diam(\mathcal P(A_4))=2$ and $diam(\mathcal P(Z_{pqr}))=3$, so it shows that the diameter of $\mathcal P(G)$ takes all the values in $\{1,2,3\}$. This completes the proof.
\end{proof}
From the above theorem, for a given finite group $G$, if  $\mathcal P(G)$ is disconnected, then $G$ has a proper subgroup $H$ with $\pi(H)=\pi(G)$. It turns out that such a subgroup $H$ will be an isolated vertex of $\mathcal P(G)$. As a consequence, we have the following result.

\begin{cor}
Let $G$ be a finite group. If $\mathcal P(G)$ is disconnected, then $\mathcal P(G) \cong \mathcal G \cup \overline{K}_r$, where $\mathcal G$ is  a connected component of $\mathcal P(G)$, and $r$ is the number of proper subgroups $H$ of $G$ with $\pi(H)=\pi(G)$.
\end{cor}

\begin{thm}\label{order prime graph of subgroups 9}
If $G$ is a finite group, then $girth(\mathcal P(G))\in \{3, 4, \infty\}$.
\end{thm}
\begin{proof}
Let $G$ be a group of order $p_1^{\alpha_1} p_2^{\alpha_2}\ldots p_k^{\alpha_k}$, where $p_i$'s are distinct primes and $\alpha_i\geq 1$. If $k\geq 3$, then any three subgroups of $G$ of distinct prime orders are mutually adjacent in $\mathcal P(G)$ and so  $\mathcal P(G)$ contains $C_3$ as a subgraph. It follows that $girth(\mathcal P(G))=3$. If $k\leq 2$, then by Corollary~\ref{c1}(2), $\mathcal P(G)$ is bipartite and so $\mathcal P(G)$ can not contain an odd cycle. Now we consider the following cases:

\noindent\textbf{Case a:} $|G|=p^\alpha q^\beta$, $\alpha$, $\beta \geq 2$. Here $G$ has subgroups of orders $p$, $p^2$, $q$, $q^2$, let them be $H_1$, $H_2$, $H_3$, $H_4$ respectively. Then $p(G)$ contains the cycle $H_1-H_3-H_2-H_4-H_1$ and so $girth(\mathcal P(G))=4$.

\noindent \textbf{Case b:} $|G|=p^\alpha q$, $\alpha\geq 2$. Suppose Sylow $q$-subgroup of $G$ is not unique, then $G$ has at least two Sylow $q$-subgroup, let them be $H_1$, $H_2$ and $G$ has subgroups of order $p$, $p^2$, let them be $H_3$, $H_4$ respectively. Then $\mathcal P(G)$ contains the cycle $H_1-H_3-H_2-H_4-H_1$ and so $girth(\mathcal P(G))=4$. Suppose Sylow $q$-subgroup of $G$ is unique, then in the bipartition of the vertex set of $\mathcal P(G)$, one partition contains only the Sylow $q$-subgroup of $G$ and another partition contains the remaining subgroups of $G$. It follows that $\mathcal P(G)$ does not contains a cycle, so $girth(\mathcal P(G))$ is $\infty$.

\noindent\textbf{Case c:} $|G|=pq$. By Theorem~\ref{order prime graph of subgroups of groups 6}(3), $\mathcal P(G)$ is a path and so $girth(\mathcal P(G))$ is $\infty$.

\noindent\textbf{Case d:} $|G|=p^\alpha$. By Corollary~\ref{c1}(1), $\mathcal P(G)$ is totally disconnected and so $girth(\mathcal P(G))$ is $\infty$.

Proof follows by combining all the above cases together.
\end{proof}

\begin{thm}\label{order prime graph of subgroups 13}
Let $n=p_1^{\alpha_1}p_2^{\alpha_2}\ldots p_k^{\alpha_k}$, where $p_i$'s are distinct primes and $\alpha_i\geq 1$. If $H$ is a proper subgroup of $\mathbb Z_n$ of order $p_{i_1}^{\alpha_{i_1}}p_{i_2}^{\alpha_{i_2}} \ldots p_{i_r}^{\alpha_{i_r}}$, then $\deg_{\mathcal P(\mathbb Z_n)}(H)= \displaystyle \prod_{\substack{j\notin \{i_1, i_2, \ldots ,i_r\}}}(\alpha_j+1)-1$.
\end{thm}
\begin{proof}
It is well known that for every divisor $d$ of $n$, $\mathbb Z_n$ has a unique subgroup of order $d$. Let $K$ be a subgroup of $\mathbb Z_n$ which is adjacent with $H$ in $\mathcal P(\mathbb Z_n)$. Then $(|H|, |K|)=1$ and  $|K|=p_{j_1}^{\alpha_{j_1}}p_{j_2}^{\alpha_{j_2}}\ldots p_{j_s}^{\alpha_{j_s}}$, with $j_1$, $j_2$, $\ldots$, $j_s\notin \{i_1, i_2, \ldots, i_r\}$. But for each $j\notin\{i_1, i_2, \ldots, i_r\}$,  the power of $p_j$ can be chosen in  $(\alpha_j+1)$ ways. It follows that, such a subgroup $K$ can be chosen in $\displaystyle \prod_{\substack{ j\notin \{i_1, i_2, \ldots ,i_r\}}}(\alpha_j+1)$ ways. Excluding the trivial subgroup,  we have $\displaystyle \prod_{\substack{ j\notin \{i_1, i_2, \ldots, i_r\}}}(\alpha_j+1)-1$ subgroups in $\mathbb Z_n$ which are adjacent with $H$ in $\mathcal P(\mathbb Z_n)$. This completes the proof.
\end{proof}

\begin{thm}\label{order prime graph of subgroups 15}
If $\mathcal G$ is a simple graph on $m$ vertices, then there exist $m'\in \mathbb N$ such that $\mathcal G$ is an induced subgraph of $\mathcal P(Z_{m'})$.
\end{thm}
\begin{proof}
 Let $n$ be the number of maximal independent sets of $\mathcal G$. Now assign $n$ distinct primes for each of these maximal independent sets. Let $v$ be a fixed vertex of $\mathcal G$. If $v$ belongs to  $t$ maximal independent sets of $\mathcal G$, then label to $v$, the product of primes which are assigned to these $t$ maximal independent sets. Similarly we can label the other vertices of $\mathcal G$.  If all these labeling are distinct, then keep them as it is. Otherwise, in order to make the labeling distinct, we relabel the vertices by using different powers of these primes.  Now let $m'$ be the least common multiple of the labels assigned to vertices of $\mathcal G$.  Again relabel each of these labels by subgroup of $\mathcal P(\mathbb{Z}_{m'})$ whose order is the same label.  Then it turns out that $\mathcal G$ is an induced subgraph of $\mathcal P(\mathbb{Z}_{m'})$. Hence the proof.
\end{proof}

Now we illustrate Theorem~\ref{order prime graph of subgroups 15} in the following example.

\begin{example}\normalfont
Consider the graph $\mathcal G$ as shown in figure~\ref{fig: f3}.
 \begin{figure}[ ht ]
\begin{center}
 \includegraphics[scale=1]{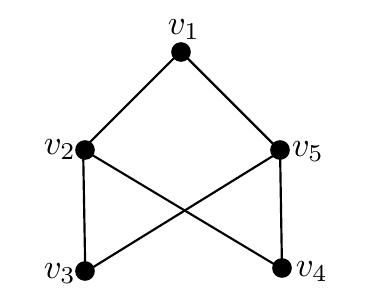}
\caption{The graph $\mathcal G$}
\label{fig: f3}
\end{center}
\end{figure}

Here $I_1:=\{v_1, v_3, v_4\}$, $I_2:=\{v_2, v_5\}$ are the only maximal independent subsets of $\mathcal G$. First we assign  prime $p_i$ to $I_i$ for each $i=1$, 2. Here $v_1\in I_1$, so label $v_1$ by $p_1$; $v_2\in I_2$, so label $v_2$ by $p_2$; $v_3\in I_1$, so label $v_3$ by $p_1$; $v_4\in I_1$, so label $v_4$ by $p_1$; $v_5\in I_2$, so label $v_5$ by $p_2$. Since $v_1$, $v_3$, $v_4$ have the same label, and $v_2$, $v_5$ have the same label, so we  relabel them in the following way: label $v_1$, $v_3$, $v_4$ by $p_1$, $p_1^2$, $p_1p_2$ respectively, and $v_2$, $v_5$ by $p_2$, $p_2^2$ respectively. Let $m'=p_1^2p_2^2$. Again relabel each of these labels by subgroup of $\mathbb{Z}_{m'}$ whose order is the same label.  It follows that $\mathcal G$ is the induced subgraph of $\mathbb{Z}_{m'}$.

\end{example}


\section{Coprime graph of subgroups of groups which are planar, unicyclic, with forbidden subgraphs }

The famous Kuratowski's Theorem (see \cite[Theorem~11.13]{harary}) characterized planar graphs as those graphs which does not contain subgraphs homeomorphic to either $K_5$ or $K_{3,3}$. We repeatedly use this theorem to show the  non-planarity of the coprime graph of subgroups of groups. The explicit plane embedding will be given when the coprime graph of subgroups of a group is planar.

The main aim of this section is to prove the following result.
\begin{thm}\label{planarity of prime graph 9}
Let $G$ be a finite group. Then
\begin{enumerate}[\normalfont (1)]
\item $\mathcal P(G)$ is planar if and only if $G$ is  one of  a $p$-group, $\mathbb Z_q\rtimes P$, $\mathbb Z_q\times  P$, where $P$ is a $p$-group, $\langle a, b, c~|~a^q=b^p=c^p=1, ac=ca, bc=cb, bab^{-1}=a^i, ord_q(i)=p\rangle$, $\mathbb Z_q\rtimes_2 \mathbb Z_{p^2}$, $\mathbb Z_{p^2}\rtimes \mathbb Z_q$, $D_{12}$, $\mathbb Z_{p^2}\rtimes \mathbb Z_{q^2}$, $\mathbb Z_{p^2}\rtimes (\mathbb Z_q\times \mathbb Z_q)$, $\mathbb Z_9\rtimes \mathbb Z_4$, $D_{18}$,  $|G|=p^\alpha q^2(\alpha\geq 3)$ with $G$ has unique cyclic Sylow  $q$-subgroup, $\mathbb Z_{pqr}$, $\mathbb Z_{p^2qr}$ or $\mathbb Z_{pqrs}$;
		
\item $\mathcal P(G)$ is $K_{2,3}$-free if and only if $G$ is one of a $p$-group, $\mathbb Z_q\rtimes P$, $\mathbb Z_q\times  P$, where $P$ is a $p$-group,  $\langle a, b, c~|~a^q=b^p=c^p=1, ac=ca, bc=cb, bab^{-1}=a^i, ord_q(i)=p\rangle$, $\mathbb Z_q\rtimes_2 \mathbb Z_{p^2}$, $D_{12}$, $\mathbb Z_{p^2q^2}$ or $\mathbb Z_{pqr}$;
		
\item $\mathcal P(G)$ is $K_{2,2}$-free if and only if $G$ is one of $p$-group, $\mathbb Z_q\times  P$, $\mathbb Z_q\rtimes P$,  where $P$ is a  $p$-group,  $\langle a, b, c~|~a^q=b^p=c^p=1, ac=ca, bc=cb, bab^{-1}=a^i, ord_q(i)=p\rangle$, $\mathbb Z_q\rtimes_2 \mathbb Z_{p^2}$, $D_{12}$ or $\mathbb Z_{pqr}$;
		
\item $\mathcal P(G)$ is $K_{1,4}$-free if and only if $G$ is one of $p$-group, $\mathbb Z_{p^\alpha q}(\alpha=1, 2, 3)$, $\mathbb Z_{p^2q^2}$, $\mathbb Z_{p^3q^2}$, $\mathbb Z_{p^3q^3}$, $S_3$ or $\mathbb Z_{pqr}$;
		
\item $\mathcal P(G)$ is $K_{1,3}$-free if and only if $G$ is one of a $p$-group, $\mathbb Z_{p^\alpha q}$ $(\alpha= 1, 2)$, $\mathbb Z_{p^2q^2}$ or $\mathbb Z_{pqr}$;
		
\item $\mathcal P(G)$ is $K_{1,2}$-free if and only if $G$ is  either a $p$-group or $\mathbb Z_{pq}$;
		
\item $\mathcal P(G)$ is unicyclic if and only if $G$ is  either $\mathbb Z_{p^2q^2}$ or $\mathbb Z_{pqr}$.
\end{enumerate}
\end{thm}

First we start with the following result:

\begin{pro}\label{planarity of prime graph 2}
If $G$ is a group whose order has atleast five distinct prime factors, then $\mathcal P(G)$ contains $K_5$.
\end{pro}
\begin{proof}
Since $G$ has atleast five prime factors, it has atleast five subgroups of distinct prime orders and so they are adjacent with each other in $\mathcal P(G)$. 
It follows that $\mathcal P(G)$ contains $K_5$.
\end{proof}

\begin{pro}\label{planarity of prime graph 3}
	Let $G$ be a group whose order has four distinct prime factors. Then 
	\begin{enumerate}[\normalfont (1)]
	\item $\mathcal P(G)$ is planar if and only if $G\cong \mathbb Z_{pqrs}$;
	\item $\mathcal P(G)$ contains $K_{2,3}$ and $K_{1,4}$.
	\end{enumerate} 
\end{pro}
\begin{proof}
	Let $|G|=p^\alpha q^\beta r^\gamma s^\delta$.	
	Let $H_1$, $H_2$, $H_3$, $H_4$ be subgroups of $G$ of order 
	$p$, $q$, $r$, $s$ respectively and let $H_5:=\langle a,b\rangle$ be subgroup of $G$, where $a$, $b$ are elements of $G$ of order $q$, $s$ respectively. Then $\mathcal P(G)$ contains $K_{1,4}$ as a subgraph with bipartition $X:=\{H_1\}$ and $Y:=\{H_2, H_3, H_4, H_5\}$. We divide the rest of the proof into two cases: 
	
	\noindent \textbf{Case 1:} Let one of $\alpha$, $\beta$, $\gamma$ or $\delta\geq 3$, without loss of generality we assume that $\alpha\geq 3$. Then $\mathcal P(G)$ contains $K_{3,3}$ as a subgraph with 
	bipartition $X:=\{H_1$, $H_2$, $H_3\}$ and $Y:=\{H_4$, $H_5$, $H_6\}$, where order of $H_1$, $H_2$, $H_3$, $H_4$, $H_5$, $H_6$ are $p$, $p^2$, $p^3$, $q$, $r$, $s$ respectively.
	
	\noindent \textbf{Case 2:} Let each $\alpha$, $\beta$, $\gamma$ and $\delta\leq 2$.
	Suppose the Sylow subgroups of $G$ is not unique, without loss of generality, we assume that the Sylow $p$-subgroup is not unique. Here $G$ has at least three subgroups, say $H_1$, $H_2$, $H_3$ of order $p^\alpha$, and subgroups $H_4$, $H_5$, $H_6$ of order $q$, $r$, $s$ respectively. Then $\mathcal P(G)$ contains $K_{3,3}$ as a subgraph with bipartition 
	$X:=\{H_1$, $H_2$, $H_3\}$ and $Y:=\{H_4$, $H_5$, $H_6\}$. 
	
	Suppose all the Sylow subgroups of $G$ are unique, then $G$ is abelian. Now we consider the following subcases:
	
	\noindent \textbf{Subase 2a:} $G$ is cyclic. If $G\cong \mathbb Z_{pqrs}$, then let $H_i$, $i=1$, 2, $\ldots$, 14 be subgroups of $G$ of orders $p$, $q$, $r$, $s$, $pq$, $pr$, $ps$, $qr$, $qs$, $rs$, $pqr$, $pqs$, $prs$, $qrs$ respectively. Then $\mathcal P(G)$ is planar and the 
	corresponding plane embedding is shown in Figure~\ref{fig: f4}; also $\mathcal P(G)$ contains $K_{2,3}$ and $K_{1,4}$.
	
	\begin{figure}[ ht ]
		\begin{center}
			\includegraphics[scale=1]{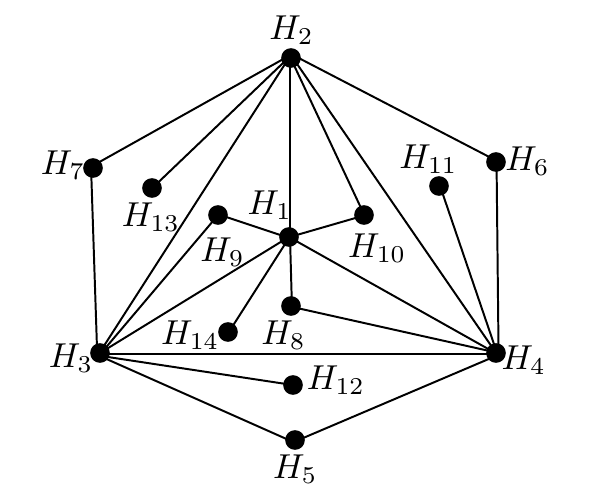}
			\caption{$\mathcal P(\mathbb Z_{pqrs}$)}
			\label{fig: f4}
		\end{center}
	\end{figure}
	
	If $G\cong \mathbb Z_{p^2qrs}$, then $\mathcal P(G)$ contains $K_{3,3}$ as a subgraph with 
	bipartition $X:=\{H_1$, $H_2$, $H_3\}$ and $Y:=\{H_4$, $H_5$, $H_6\}$, where $H_1$, $H_2$, $H_3$, $H_4$, $H_5$, $H_6$ are subgroups of $G$ of order 
	$p$, $p^2$, $pq$, $r$, $s$, $rs$ respectively. If $G\cong \mathbb Z_{p^2q^2rs}$, $\mathbb Z_{p^2q^2r^2s}$ or $\mathbb Z_{p^2q^2r^2s^2}$, then $G$ has 
	$\mathbb Z_{p^2qrs}$ as a subgroup and so $\mathcal P(G)$ contains $K_{3,3}$.
	
	\noindent \textbf{Subcase 2b:} $G$ is non-cyclic abelian. Here $G$ has atleast three subgroups of any one of prime order $p$, $q$, $r$ or $s$, without loss of generality, say $p$. 
	Then $\mathcal P(G)$ contains $K_{3,3}$ as a subgraph with 
	bipartition $X:=\{H_1$, $H_2$, $H_3\}$ and $Y:=\{H_4$, $H_5$, $H_6\}$, where $H_1$, $H_2$, $H_3$, $H_4$, $H_5$, $H_6$ are subgroups of $G$ of order 
	$p$, $p$, $p$, $q$, $r$, $s$ respectively.
	
	Combining all the cases together the proof follows.
\end{proof}

\begin{pro}\label{order prime graph of subgroups of groups 11}
	Let $G$ be a group whose order has three distinct prime factors. Then
	\begin{enumerate}[\normalfont (1)]
		\item $\mathcal P(G)$ is planar if and only if $G$  is isomorphic to either $\mathbb Z_{pqr}$ or $\mathbb Z_{p^2qr}$;
		\item 	$\mathcal P(G)$ contains $K_{1,2}$;
		\item The following are equivalent:
		\begin{enumerate}[\normalfont (a)]
			\item $G\cong \mathbb Z_{pqr}$;
			\item $\mathcal P(G)$ is $K_{1,4}$-free;
\item $\mathcal P(G)$ is $K_{2,3}$-free;
			
			\item $\mathcal P(G)$ is $K_{1,3}$-free;
			\item $\mathcal P(G)$ is $K_{2,2}$-free;
			\item $\mathcal P(G)$ is unicyclic.
		\end{enumerate}
	\end{enumerate}

\end{pro}
\begin{proof}
	Let $|G|=p^\alpha q^\beta r^\gamma$. We need to consider the following cases.
	
	\noindent \textbf{Case 1:} Let one of $\alpha$, $\beta$ or $\gamma\geq 3$, without loss of generality we assume that $\alpha\geq 3$. Then $\mathcal P(G)$ contains $K_{3,3}$ as a subgraph with 
	bipartition $X:=\{H_1$, $H_2$, $H_3\}$ and $Y:=\{H_4$, $H_5$, $H_6\}$, where $H_1$, $H_2$, $H_3$, $H_4$, $H_5$ are subgroups of $G$ of order 
	$p$, $p^2$, $p^3$, $q$, $r$ respectively and $H_6:=\langle a,b\rangle$, $a$, $b$ are elements of $G$ of order $q$, $r$ respectively. Moreover, $\mathcal P(G)$ contains $K_{1,4}$ as a subgraph with 
	bipartition $X:=\{H_1$, $H_2$, $H_3$, $K_2\}$ and $Y:=\{H_4\}$.
	
	\noindent \textbf{Case 2:} Let each $\alpha$, $\beta$ and $\gamma\leq 2$.
	Suppose a Sylow subgroup of $G$ is not unique, without loss of generality, we assume that Sylow $p$-subgroup is not unique. Then $G$ has at least three subgroups, say $H_1$, $H_2$, $H_3$ of order $p^\alpha$, and subgroups
	$H_4$, $H_5$ of order $q$, $r$, respectively, and  $H_6:=\langle a,b\rangle$, $a$, $b$ are elements of $G$ of order $q$, $r$ respectively. Then $\mathcal P(G)$ contains 
	$K_{3,3}$ as a subgraph with bipartition $X:=\{H_1$, $H_2$, $H_3\}$ and $Y:=\{H_4$, $H_5$, $H_6\}$. Also $\mathcal P(G)$ contains $K_{1,4}$ as a subgraph with 
	bipartition $X:=\{H_1$, $H_2$, $H_3$, $K_2\}$ and $Y:=\{H_4\}$.
	
	Now assume that all the Sylow subgroups of $G$ are unique. Then $G$ is abelian.
	\begin{enumerate}[{\normalfont (a)}]
    \item [(a)] Let $G$ be cyclic. If $G\cong \mathbb Z_{pqr}$, then let $H_i$, $i=1$, 2, $\ldots$, 6 be subgroups of $G$ of orders $p$, $q$, $r$, $pq$, $pr$, $qr$ respectively. Then $\mathcal P(G)$ is planar and the 
	corresponding plane embedding is shown in Figure~\ref{fig: f1}. 
	\begin{figure}[ht]
	\begin{center}
	\includegraphics[scale =0.6]{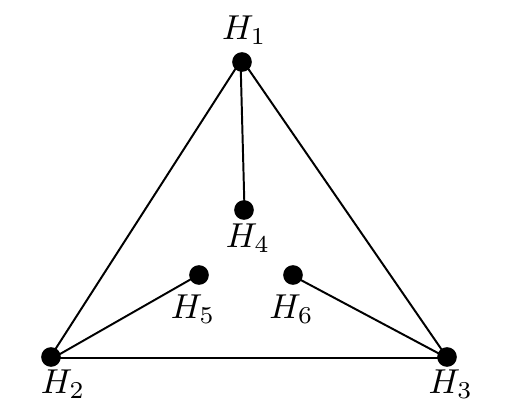}
	\caption{$\mathcal P(\mathbb Z_{pqr})$}
	\label{fig: f1}
	\end{center}
	\end{figure}
	Also $\mathcal P(G)$ is unicyclic; it does not contain  $K_{2,2}$, $K_{1,3}$ as subgraphs; but it contains $K_{1,2}$ as a subgraph. If $G\cong \mathbb Z_{p^2qr}$, then let $H_i$, $i=1$, 2, $\ldots$, 10 be subgroups of $G$ of orders $p$, $p^2$, $q$, $r$, $pq$, $pr$, $qr$, $p^2q$, $p^2r$, $pqr$ respectively. Then $\mathcal P(G)$ is planar and the 
	corresponding plane embedding is shown in Figure~\ref{fig:f5}.
	\begin{figure}[ ht ]
		\begin{center}
			\begin{minipage}{.4\textwidth}
				\begin{center}
					\includegraphics[scale =0.8]{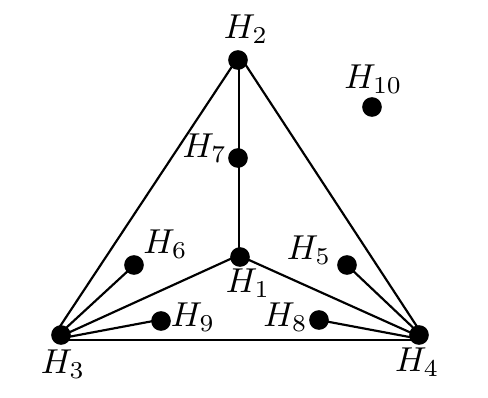}
					\caption{$\mathcal P(\mathbb Z_{p^2qr})$}
					\label{fig:f5}
				\end{center}
			\end{minipage}
			\begin{minipage}{.4\textwidth}
				\begin{center}
					\includegraphics[scale =0.8]{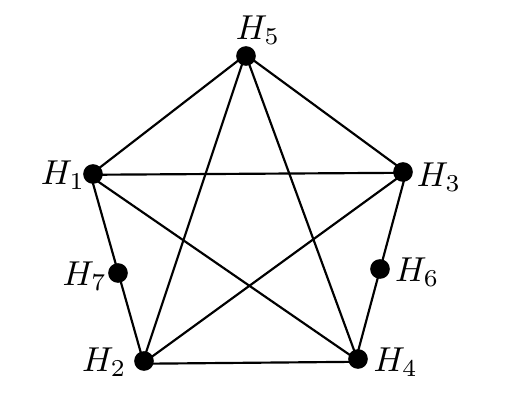}
					\caption{A subdivision of $K_5$}
					\label{fig:f6}
				\end{center}
			\end{minipage}
		\end{center}
	\end{figure}
	%
	Also $\mathcal P(G)$ is not unicyclic; it contains $K_{2,3}$ as a subgraph and does not contain $K_{1,4}$. If $G\cong \mathbb Z_{p^2q^2r}$ or $\mathbb Z_{p^2q^2r^2}$, then 
	$\mathcal P(G)$ contains a subdivision of $K_5$ as shown in Figure~\ref{fig:f6} with vertices $H_i$, $i=1$, $\ldots$, 7 of order 
	$p$, $p^2$, $q$, $q^2$, $r$, $pr$, $qr$ respectively. Also $\mathcal P(G)$ contains $K_{1,4}$ as a subgraph with 
	bipartition $X:=\{H_1$, $H_2$, $H_3$, $H_4\}$ and $Y:=\{H_5\}$; it contains $K_{2,3}$ as a subgraph with bipartition $X:=\{H_1$, $H_2\}$ and $Y:=\{H_3, H_4, H_5\}$.
	
	\item [(b)] Let $G$ be non-cyclic abelian. Here $G$ has atleast three subgroups of any one of prime order $p$, $q$, $r$ or $s$, let it be $p$. 
	Then $\mathcal P(G)$ contains $K_{3,3}$ as a subgraph with 
	bipartition $X:=\{H_1$, $H_2$, $H_3\}$ and $Y:=\{H_4$, $H_5$, $H_6\}$, where $H_1$, $H_2$, $H_3$, $H_4$, $H_5$, $H_6$ are subgroups of $G$ of order 
	$p$, $p$, $p$, $q$, $r$, $s$ respectively. Also $\mathcal P(G)$ contains $K_{1,4}$ as a subgraph with 
	bipartition $X:=\{H_1$, $H_2$, $H_3$, $H_5\}$ and $Y:=\{H_4\}$.
\end{enumerate}	
	The proof follows by combining all the above cases together.
\end{proof}

\begin{pro}\label{planarity of prime graph 4}
	Let $G$ be a group of order $p^\alpha q^\beta$, $\alpha$, $\beta\geq 2$, $\alpha+\beta\geq 6$. Then 
	\begin{enumerate}[\normalfont (1)]
		\item $\mathcal P(G)$ is planar if and only if $|G|=p^\alpha q^2$ with $G$ has  unique cyclic Sylow  $q$-subgroup;
		\item $\mathcal P(G)$ is $K_{1,4}$-free if and only if $G\cong \mathbb Z_{p^3q^3}$;
	\item $\mathcal P(G)$ constains  $K_{2,3}$.
	\end{enumerate}
	
\end{pro}
\begin{proof}
	Proof is divided into two cases:
	
	\noindent \textbf{Case 1:} $\alpha+\beta\geq 6$ and $\alpha$, $\beta\geq 3$.
	Let $H_1$, $H_2$, $H_3$, $H_4$, $H_5$, $H_6$ be subgroups of $G$ of order $p$, $p^2$, $p^3$, $q$, $q^2$, $q^3$ respectively. Then $\mathcal P(G)$ 
	contains $K_{3,3}$ as a subgraph with bipartition $X:=\{H_1, H_2, H_3\}$ and $Y:=\{H_4, H_5, H_6\}$.
	
	If $\alpha\geq 4$, then $G$ has subgroups of order $p$, $p^2$, $p^3$, $p^4$, $q$, let them be $H_1$, $H_2$, $H_3$, $H_4$, $H_5$ respectively. Then $\mathcal P(G)$ contains $K_{1,4}$ as a subgraph with bipartition $X:=\{H_1, H_2, H_3, H_4\}$ and $Y:=\{H_5\}$.
	
	Let $\alpha=3$. Suppose Sylow $p$-subgroup and Sylow $q$-subgroup of $G$ are cyclic and are unique, then $G$ is cyclic and so $\mathcal P(G)$ is bipartite with one partition contains all subgroups whose orders has $p$ as a divisor and another partition contains remaining proper subgroups of $G$. Thus $\mathcal P(G)\cong K_{3,3}\cup \overline{K}_{10}$. It follows that $\mathcal P(G)$ is $K_{1,4}$-free. 
	
	If Sylow $p$-subgroup or Sylow $q$-subgroup of $G$ is not cyclic, without loss of generality, we assume that Sylow $p$-subgroup is not cyclic, then $G$ has subgroups of order $p$, $p^2$, $p^3$, $p^i$, $q$, for some $i\in\{1, 2\}$, let them be $H_1$, $H_2$, $H_3$, $H_4$, $H_5$ respectively. Then $\mathcal P(G)$ contains $K_{1,4}$ as a subgraph with bipartition $X:=\{H_1, H_2, H_3, H_4\}$ and $Y:=\{H_5\}$.
	
	If Sylow $p$-subgroup or Sylow $q$-subgroup is not unique, without loss of generality, we assume that Sylow $p$-subgroup is not unique, then $G$ has subgroups of order $p$, $p^2$, $p^3$, $p^3$, $q$, let them be $H_1$, $H_2$, $H_3$, $H_4$, $H_5$ respectively. Then $\mathcal P(G)$ contains $K_{1,4}$ as a subgraph with bipartition $X:=\{H_1, H_2, H_3, H_4\}$ and $Y:=\{H_5\}$. 
	
	\noindent \textbf{Case 2}: $\alpha+\beta\geq 6$, $\beta=2$. 
	Let $H_1$, $H_2$, $H_3$, $H_4$, $H_5$, $H_6$ be subgroups of $G$ of order $p$, $p^2$, $p^3$, $p^4$, $q^2$, $q$ respectively. Then $\mathcal P(G)$ 
	contains $K_{1,4}$ as a subgraph with bipartition $X:=\{H_1, H_2, H_3, H_4\}$ and $Y:=\{H_5\}$; also $\mathcal P(G)$ contains $K_{2,3}$ as a subgraph with bipartition $X:=\{H_1, H_2, H_3\}$ and $Y:=\{H_5, H_6\}$. 
	
	If Sylow $q$-subgroup of $G$ is not unique, then $G$ has subgroups of order $p$, $p^2$, $p^3$, $q^2$, $q^2$, $q$, let them be $H_1$, $H_2$, $H_3$, $H_4$, $H_5$, $H_6$ respectively. Then $\mathcal P(G)$ contains $K_{3,3}$ as a subgraph with bipartition $X:=\{H_1, H_2, H_3\}$ and $Y:=\{H_4, H_5, H_6\}$.
	
	If Sylow $q$-subgroup of $G$ is unique and is isomorphic to $\mathbb Z_q\times \mathbb Z_q$, then $G$ has subgroups of order $p$, $p^2$, $p^3$, $q^2$, $q$, $q$, let them be $H_1$, $H_2$, $H_3$, $H_4$, $H_5$, $H_6$ respectively. Then $\mathcal P(G)$ contains $K_{3,3}$ as a subgraph with bipartition $X:=\{H_1, H_2, H_3\}$ and $Y:=\{H_4, H_5, H_6\}$.
	
	If Sylow $q$-subgroup of $G$ is unique and cyclic, then $G$ has an unique subgroups of order $q$, $q^2$, let them be $H_1$, $H_2$ respectively. Then $\mathcal P(G)$ is planar, since $\mathcal P(G)$ is bipartite with one partition contains all the subgroups whose order has $p$ as a divisor and another partition contains $H_1$, $H_2$.
	
	The proof follows by combining the above cases together.
\end{proof}

\begin{pro}\label{planarity of prime graph 5}
	Let $G$ be a group of order $p^3q^2$. Then 
	\begin{enumerate} [\normalfont (1)]
		\item $\mathcal P(G)$ is planar if and only if $G$ has a unique cyclic Sylow $q$-subgroup;
		\item $\mathcal P(G)$ is $K_{1,4}$-free if and only if $G\cong \mathbb Z_{p^3q^2}$.
		\item $\mathcal P(G)$ constains  $K_{2,3}$.
	\end{enumerate}
	
\end{pro}
\begin{proof}
	Let $H_1$, $H_2$, $H_3$, $H_4$, $H_5$ be subgroups of $G$ of order $p$, $p^2$, $p^3$, $q$, $q^2$ respectively. Here $\mathcal P(G)$ contains $K_{3,2}$ as a subgraph with bipartition $X:=\{H_1, H_2, H_3\}$ and $Y:=\{H_4, H_5\}$. 
	
	Suppose Sylow $p$-subgroup, and Sylow $q$-subgroup of $G$ are cyclic, and are unique, then $G$ is cyclic and so $\mathcal P(G)$ is bipartite with one partition contains all subgroups whose order has $p$ as a divisor and another partition contains remaining subgroups of $G$. Thus $\mathcal P(G)\cong K_{2,3}\cup \overline{K}_7$. It follows that $\mathcal P(G)$ is $K_{1,4}$-free. 
	
	If Sylow $p$-subgroup of $G$ is not unique, then $G$ has subgroups of order $p^3$, $p^3$, $p^2$, $p$, $q^2$, say $H_1$, $H_2$, $H_3$, $H_4$, $H_5$ respectively. Then $\mathcal P(G)$ contains $K_{1,4}$ as a subgraph with bipartition $X:=\{H_1, H_2, H_3, H_4\}$ and $Y:=\{H_5\}$. 
	
	If Sylow $p$-subgroup of $G$ is not cyclic, then $G$ has subgroups of order $p^3$, $p^i$, $p^2$, $p$, $q^2$, for some $i\in\{1, 2\}$, say $H_1$, $H_2$, $H_3$, $H_4$, $H_5$ respectively. Then $\mathcal P(G)$ contains $K_{1,4}$ as a subgraph with bipartition $X:=\{H_1, H_2, H_3, H_4\}$ and $Y:=\{H_5\}$. 
	
	Suppose Sylow $q$-subgroup of $G$ is isomorphic to $\mathbb Z_q\times \mathbb Z_q$, then $G$ has subgroups of order $p^3$, $p^2$, $p$, $q$, $q$, $q$, $q^2$, say $H_1$, $H_2$, $H_3$, $H_4$, $H_5$, $H_6$, $H_7$ respectively. Then  $\mathcal P(G)$ contains $K_{3,3}$ as a subgraph with bipartition $X:=\{H_1, H_2, H_3\}$ and $Y:=\{H_4, H_5, H_6\}$ and $\mathcal P(G)$ contains $K_{1,4}$ as a subgraph with bipartition $X:=\{H_4, H_5, H_6, H_7\}$ and $Y:=\{H_1\}$. 
	
	Suppose Sylow $q$-subgroup of $G$ is not unique, then $G$ has subgroups of order $p^3$, $p^2$, $p$, $q$, $q^2$, $q^2$, $q^2$, say $H_1$, $H_2$, $H_3$, $H_4$, $H_5$, $H_6$, $H_7$ respectively. Then $\mathcal P(G)$ contains $K_{3,3}$ as a subgraph with bipartition $X:=\{H_1, H_2, H_3\}$ and $Y:=\{H_4, H_5, H_6\}$ and it contains $K_{1,4}$ as a subgraph with bipartition $X:=\{H_4, H_5, H_6, H_7\}$ and $Y:=\{H_1\}$. 
	
	If Sylow $q$-subgroup of $G$ is unique and it is cyclic, then $G$ has an unique subgroup of order $q$, $q^2$, let them be $H_1$, $H_2$ respectively. Then $\mathcal P(G)$ is planar, since $\mathcal P(G)$ is bipartite with one partition contains all the subgroups whose order has $p$ as a divisor and another partition contains $H_1$, $H_2$. It follows that $\mathcal P(G)$ is planar if and only if $G$ has a unique cyclic Sylow $q$-subgroup. In \cite{tripp}  Myron Owen Tripp showed that up to isomophism, there are 15 such groups of order  $p^3q^2$ exist. 
\end{proof}

\begin{pro}\label{planarity of prime graph 6}
Let $G$ be a group of order $p^2q^2$. Then 
\begin{enumerate}[\normalfont (1)]
\item $\mathcal P(G)$ is planar if and only if $G$ is isomorphic to one of $\mathbb Z_{p^2}\rtimes \mathbb Z_{q^2}$, $\mathbb Z_{p^2}\rtimes (\mathbb Z_q\times \mathbb Z_q)$, $\mathbb Z_9\rtimes \mathbb Z_4$ or $D_{18}$;

\item $\mathcal P(G)$ contains $K_{2,2}$;
\item The following are equivalent:
\begin{enumerate}[\normalfont (a)]
\item $G\cong \mathbb Z_{p^2q^2}$;
\item $\mathcal P(G)$ is $K_{1,4}$-free;
\item $\mathcal P(G)$ is $K_{2,3}$-free;
\item $\mathcal P(G)$ is $K_{1,3}$-free;
\item $\mathcal P(G)$ is unicyclic.
\end{enumerate} 
\end{enumerate} 

\end{pro}
\begin{proof}
We consider the following cases: 
	
\noindent \textbf{Case 1:} $G$ is non-abelian. We use the classification of groups of order $p^2q^2$ given in \cite{lin}. With out loss of generality, we assume that $p>q$. 
Let $P$ and $Q$ denote a Sylow $p$-subgroup and Sylow $q$-subgroup of $G$ respectively. By Sylow's theorem, $n_p=1, q, q^2$. But $n_p= q$ is not possible, since $p> q$. If $n_p= q^2$, then $p~|~(q+1)(q-1)$, this implies that $p~|~(q+1)$, which is true only
when $(p, q)= (3, 2)$.  When $(p, q) \neq (3, 2)$, then $G \cong P \rtimes Q$.
	
If $G \cong \mathbb Z_{p^2} \rtimes \mathbb Z_{q^2}= \langle a, b| a^{p^2}=b^{q^2}=1, bab^{-1}=a^i, i^{q^2} \equiv 1 (\mbox{mod}~ p^2) \rangle$, then 
$\mathcal P(G)$ is planar, since it is bipartite with one partition contains all the subgroups of $G$ whose order has $q$ as a divisor, and another partition contains the subgroups of $G$ order $p$, $p^2$; also $G$ has $p^2$ Sylow $q$-subgroups. It follows that $\mathcal P(G)$ contains $K_{2,3}$ and $K_{1,4}$ as subgraphs and so $\mathcal P(G)$ is not unicyclic.
	
If $G \cong \mathbb Z_{p^2} \times (\mathbb Z_q \times \mathbb Z_q)$, then  
$\mathcal P(G)$ is planar, since it is bipartite with one partition contains all the subgroups whose order has $q$ as a divisor, and another partition contains the subgroups of $G$ order $p$, $p^2$; also $G$ has $p^2$ Sylow $q$-subgroups. It follows that $\mathcal P(G)$ contains $K_{2,3}$ and $K_{1,4}$ as a subgraph and so $\mathcal P(G)$ is not unicyclic.
	
If $G \cong (\mathbb Z_p \times \mathbb Z_p) \rtimes \mathbb Z_{q^2}:=\langle a,b,c~|~a^p=b^p=c^{q^2}=1, ab=ba, cac^{-1}=a^ib^j, cbc^{-1}=a^kb^l\rangle$,
where $\bigl(\begin{smallmatrix}
i & j\\ k & l
\end{smallmatrix} \bigr)$ has order $q^2$ in $GL_2(p)$, then $\mathcal P(G)$ contains $K_{3,3}$ as a subgraph with bipartition 
$X:=\{H_1$, $H_2$, $H_3\}$ and $Y:=\{H_4$, $H_5$, $H_6\}$, where $H_1$, $H_2$, $H_3$, $H_4$, $H_5$, $H_6$ are subgroups of $G$ of order $p$, $p$, $p$, 
$q^2$, $q^2$, $q$ respectively. Also $G$ has $p^2$ Sylow $q$-subgroups, so $\mathcal P(G)$ contains $K_{1,4}$ as a subgraph.
	
If $G \cong (\mathbb Z_p \times \mathbb Z_p) \rtimes (\mathbb Z_q \times \mathbb Z_q)$,  
Then $\mathcal P(G)$ contains $K_{3,3}$ as a subgraph with bipartition 
$X:=\{H_1$, $H_2$, $H_3\}$ and $Y:=\{H_4$, $H_5$, $H_6\}$, where $H_1$, $H_2$, $H_3$, $H_4$, $H_5$, $H_6$ are subgroups of $G$ of order $p$, $p$, $p$, 
$q$, $q$, $q$ respectively. Also $G$ has $p^2$ Sylow $q$-subgroups, so $\mathcal P(G)$ contains $K_{1,4}$ as a subgraph.
	
\noindent If $(p, q)= (3, 2)$, then up to isomorphism, there are nine groups of order 36. In the following we consider each of these groups.
\begin{enumerate}[{\normalfont (a)}]
\item $G\cong D_{18}$. Here $G$ has unique subgroups of order 3 and 9 respectively; the order of remaining proper subgroups have 2 as their common divisor. Therefore, $\mathcal P(G)$ is planar, since it is bipartite with one partition contains  subgroups of order 3, 9, and another partition contains remaining proper subgroups of $G$. Since there are 18 subgroups of $D_{18}$ has order 2, so $\mathcal P(G)$ contains $K_{2,3}$ and $K_{1,4}$ as a subgraph.
		
\item $G\cong S_3\times S_3$. Let $H_1$, $H_2$, $H_3$, $H_4$, $H_5$, $H_6$, $H_7$ be subgroups of $G$ of order 3, 3, 9, 2, 2, 2, 2 respectively. Then $\mathcal P(G)$ contains $K_{3,3}$ as a subgraph with bipartition $X:=\{H_1, H_2, H_3\}$ and $Y:=\{H_4, H_5, H_6\}$. Also 
$\mathcal P(G)$ contains $K_{1,4}$ as a subgraph with bipartition $X:=\{H_4, H_5, H_6, H_7\}$ and $Y:=\{H_1\}$.
		
\item $G\cong \mathbb Z_3\times A_4$. Let $H_1$, $H_2$, $H_3$, $H_4$, $H_5$, $H_6$, $H_7$ be subgroups of $G$ of order 3, 3, 9, 2, 2, 2, 4 respectively. Then $\mathcal P(G)$ contains $K_{3,3}$ as a subgraph with bipartition $X:=\{H_1, H_2, H_3\}$ and $Y:=\{H_4$, $H_5$, $H_6\}$. Also 
$\mathcal P(G)$ contains $K_{1,4}$ as a subgraph with bipartition $X:=\{H_4, H_5, H_6, H_7\}$ and $Y:=\{H_1\}$.
		
\item $G\cong \mathbb Z_6\times S_3$. Let $H_1$, $H_2$, $H_3$, $H_4$, $H_5$, $H_6$, $H_7$ be subgroups of $G$ of order 3, 3, 9, 2, 2, 2, 2 respectively. Here $\mathcal P(G)$ contains $K_{3,3}$ as a subgraph with bipartition $X:=\{H_1, H_2, H_3\}$ and $Y:=\{H_4$, $H_5$, $H_6\}$. Also 
$\mathcal P(G)$ contains $K_{1,4}$ as a subgraph with bipartition $X:=\{H_4, H_5, H_6, H_7\}$ and $Y:=\{H_1\}$.
		
\item $G\cong \mathbb Z_9\rtimes\mathbb Z_4=\langle a,b~|~a^9=b^4=1, bab^{-1}=a^i, i^4\equiv 1(\mbox{mod}~ 9)\rangle$, then 
$\mathcal P(G)$ is planar, since it is bipartite with one partition contains all the subgroups whose order has $p$ as a divisor, and another partition contains the subgroups of order $q$, $q^2$. It follows that $\mathcal P(G)$ contains $K_{2,3}$ as a subgraph. Also $G$ has nine subgroups of order 4. So $\mathcal P(G)$ contains $K_{1,4}$ as a subgraph with bipartition $X:=\{H_1, H_2, H_3, H_4\}$ and $Y:=\{H_5\}$, where $H_1$, $H_2$, $H_3$, $H_4$, $H_5$ are subgroups of $G$ of order 4, 4, 4, 2, 3 respectively.
		
\item $G\cong \mathbb Z_3\times(\mathbb Z_3\rtimes \mathbb Z_4)=\langle a,b,c~|~a^3=b^3=c^4=1, ab=ba, ac=ca, cbc^{-1}=b^i,
\mbox{ord}_2(i)=3\rangle$. Let $H_1$, $H_2$, $H_3$, $H_4$, $H_5$, $H_6$, $H_7$ be subgroups of $G$ of order 3, 3, 9, 2, 4, 4, 4 respectively. Then $\mathcal P(G)$ contains $K_{3,3}$ as a subgraph with bipartition $X:=\{H_1, H_2, H_3\}$ and $Y:=\{H_4$, $H_5$, $H_6\}$. Also 
$\mathcal P(G)$ contains $K_{1,4}$ as a subgraph with bipartition $X:=\{H_4, H_5, H_6, H_7\}$ and $Y:=\{H_1\}$.
		
\item $G\cong (\mathbb Z_3\times \mathbb Z_3)\rtimes \mathbb Z_4:=\langle a,b,c~|~a^3=b^3=c^4=1, ab=ba, cac^{-1}=a^ib^j, cbc^{-1}=a^kb^l\rangle$,
where $\bigl(\begin{smallmatrix}
i & j\\ k & l
\end{smallmatrix} \bigr)$ has order $4$ in $GL_2(3)$. Let $H_1$, $H_2$, $H_3$, $H_4$, $H_5$, $H_6$, $H_7$ be subgroups of $G$ of order 3, 3, 3, 9, 2, 4, 4 respectively. Here $\mathcal P(G)$ contains $K_{3,3}$ as a subgraph with bipartition $X:=\{H_1, H_2, H_3\}$ and $Y:=\{H_5$, $H_6$, $H_7\}$. Also 
$\mathcal P(G)$ contains $K_{1,4}$ as a subgraph with bipartition $X:=\{H_1, H_2, H_3, H_4\}$ and $Y:=\{H_5\}$.
		
\item $G\cong (\mathbb Z_2\times (\mathbb Z_3\times \mathbb Z_3))\rtimes \mathbb Z_2$. Let $H_1$, $H_2$, $H_3$, $H_4$, $H_5$, $H_6$, $H_7$ be subgroups of $G$ of order 3, 3, 3, 9, 2, 2, 2 respectively. Here $\mathcal P(G)$ contains $K_{3,3}$ as a subgraph with bipartition $X:=\{H_1, H_2, H_3\}$ and $Y:=\{H_5$, $H_6$, $H_7\}$. Also 
$\mathcal P(G)$ contains $K_{1,4}$ as a subgraph with bipartition $X:=\{H_1, H_2, H_3, H_4\}$ and $Y:=\{H_5\}$.
		
\item $G\cong (\mathbb Z_2\times \mathbb Z_2)\rtimes \mathbb Z_9$. Let $H_1$, $H_2$, $H_3$, $H_4$, $H_5$, $H_6$, $H_7$ be subgroups of $G$ of order 3, 3, 9, 2, 2, 2, 4 respectively. Here $\mathcal P(G)$ contains $K_{3,3}$ as a subgraph with bipartition $X:=\{H_1, H_2, H_3\}$ and $Y:=\{H_4$, $H_5$, $H_6\}$. Also 
$\mathcal P(G)$ contains $K_{1,4}$ as a subgraph with bipartition $X:=\{H_4, H_5, H_6, H_7\}$ and $Y:=\{H_5\}$.
\end{enumerate}
	
\noindent \textbf{Case 2:} $G$ is abelian. If $G\cong \mathbb Z_{p^2q^2}$, then it is easy to see that $\mathcal P(G)\cong K_{2,2}\cup \overline{K}_3$, which 
 is planar, $K_{1,3}$-free and it contains $K_{2,2}$ as a subgraph. 

If $G\cong \mathbb Z_{p^2q}\times \mathbb Z_q$, 
then $\mathcal P(G)$ is planar, since $\mathcal P(G)$ is bipartite with one partition contains all the subgroups whose order has $p$ as a divisor, and another partition contains the subgroups of order $q$, $q^2$. It follows that $\mathcal P(G)$ contains $K_{2,3}$ as a subgraph. Also $\mathcal P(G)$ contains $K_{1,4}$ as a subgraph with bipartition $X:=\{H_1, H_2, H_3, H_4\}$ and $Y:=\{H_5\}$, where $H_1$, $H_2$, $H_3$, $H_4$, $H_5$ are subgroups of $G$ of order $q$, $q$, $q$, $q^2$, $p$ respectively.
	
If $G\cong \mathbb Z_{pq}\times \mathbb Z_{pq}$, then $\mathcal P(G)$ contains $K_{3,3}$ as a subgraph with bipartition 
$X:=\{H_1$, $H_2$, $H_3\}$ and $Y:=\{H_4$, $H_5$, $H_6\}$, where $H_1$, $H_2$, $H_3$, $H_4$, $H_5$, $H_6$, $H_7$ are subgroups of $G$ of order $p$, $p$, $p$, 
$q$, $q$, $q$, $q^2$ respectively. Also $\mathcal P(G)$ contains $K_{1,4}$ as a subgraph with bipartition $X:=\{H_4, H_5, H_6, H_7\}$ and $Y:=\{H_1\}$.

The proof follows by combining all the  above cases together.
\end{proof}

\begin{pro}\label{planarity of prime graph 7}
Let $G$ be a group of order $p^\alpha q$. Then
\begin{enumerate}[\normalfont (1)]
\item $\mathcal P(G)$ is planar if and only if $G$ is one of $\mathbb Z_q\times P$, $\mathbb Z_q\rtimes P$,  where $P$ is a  $p$-group, $\langle a, b, c~|~a^q=b^p=c^p=1, ac=ca, bc=cb, bab^{-1}=a^i, ord_q(i)=p\rangle$, $\mathbb Z_q\rtimes_2 \mathbb Z_{p^2}$, $\mathbb Z_{p^2}\rtimes \mathbb Z_q$ or $D_{12}$;
\item The following are equivalent:
\begin{enumerate}[\normalfont (a)]
\item  $G$ is  one of $\mathbb Z_q\times P$, $\mathbb Z_q\rtimes P$,  where $P$ is a  $p$-group, $\langle a, b, c~|~a^q=b^p=c^p=1, ac=ca, bc=cb, bab^{-1}=a^i, ord_q(i)=p\rangle$, $\mathbb Z_q\rtimes_2 \mathbb Z_{p^2}$ or $D_{12}$;
\item $\mathcal P(G)$ is $K_{2,3}$-free; 
\item $\mathcal P(G)$ is $K_{2,2}$-free.
\end{enumerate}
\item $\mathcal P(G)$ is $K_{1,4}$-free if and only if $G$ is  one of $\mathbb Z_{p^\alpha q}(\alpha=1, 2, 3)$ or $S_3$.
\item $\mathcal P(G)$ is $K_{1,3}$-free if and only if $G$ either $\mathbb Z_{pq}$ or $\mathbb Z_{p^2q}$;
\item $\mathcal P(G)$ is $K_{1,2}$-free if and only if $G \cong \mathbb Z_{pq}$;
\item $\mathcal P(G)$ is acyclic.
\end{enumerate}  

\end{pro}
\begin{proof} Proof is divided into several cases.\\
\noindent \textbf{Case 1:} $\alpha\geq 3$.

 Suppose Sylow $q$-subgroup of $G$ is not unique, then $G$ has subgroups of order $p$, $p^2$, $p^3$, $q$, $q$, $q$, let them be $H_1$, $H_2$, $H_3$, $H_4$, $H_5$, $H_6$ respectively. Here $\mathcal P(G)$ contains $K_{3,3}$ as a subgraph with bipartition $X:=\{H_1, H_2, H_3\}$ and $Y:=\{H_4, H_5, H_6\}$.
	
Suppose Sylow $q$-subgroup is unique, then $G\cong Z_p\times P $ or $Z_q\rtimes P$, where $P$ is a Sylow $p$-subgroup of $G$ and so $\mathcal P(G)$ is bipartite with one partition contains a subgroup whose order is $q$ and another partition contains remaining proper subgroups of $G$. Hence $\mathcal P(G)$ is planar and $K_{2,2}$-free. 

Now we check the $K_{1,4}$-freeness of $\mathcal P(G)$. If $\alpha\geq 4$, then $G$ has subgroups of order $p$, $p^2$, $p^3$, $p^4$, $q$, let them be $H_1$, $H_2$, $H_3$, $H_4$, $H_5$ respectively. Then $\mathcal P(G)$ contains $K_{1,4}$ as a subgraph with bipartition $X:=\{H_1, H_2, H_3, H_4\}$ and $Y:=\{H_5\}$.

Let $\alpha=3$. Suppose Sylow $p$-subgroup of $G$ is not unique, then $G$ has subgroups of order $p$, $p^2$, $p^3$, $p^3$, $q$, let them be $H_1$, $H_2$, $H_3$, $H_4$, $H_5$ respectively. Then $\mathcal P(G)$ contains $K_{1,4}$ as a subgraph with bipartition $X:=\{H_1, H_2, H_3, H_4\}$ and $Y:=\{H_5\}$.

Suppose Sylow $p$-subgroup of $G$ is unique and cyclic, then $G$ has a unique subgroups of order $p$, $p^2$, $p^3$, let them be $H_1$, $H_2$, $H_3$, and so $\mathcal P(G)$ is a bipartite graph with one partition contains $H_1$, $H_2$, $H_3$ and another partition contains remaining proper subgroups of $G$. Hence $\mathcal P(G)$ is $K_{1,4}$-free. But $\mathcal P(G)$ contains $K_{1,3}$ as a subgraph with bipartition $X:=\{H_1, H_2, H_3\}$ and $Y:=\{H_4\}$, where $H_4$ is a subgroup of order $q$.
	
\noindent \textbf{Case 2:} $\alpha=2$, $\beta=1$. 
	
Suppose $G$ is cyclic, then $G$ has only four proper subgroups and their orders are $p^2$, $p$, $q$, $pq$ respectively. Therefore, $\mathcal P(G)\cong K_{1,2}\cup K_1$, which is planar and $K_{1,2}$-free. 
	
If $G$ is non-cyclic abelian, then $G$ has $p+1$ subgroups of order $p$; unique subgroup of order $p^2$; unique subgroup of order $q$; $p+1$ subgroups of order $pq$. Also these are the only proper subgroups of $G$. It follows that $\mathcal P(G)\cong K_{1,p+2}\cup \overline{K}_{p+1}$ and hence $\mathcal P(G)$ is planar; it contains $K_{1,4}$ as a subgraph and it is $K_{2,2}$-free.
	
If $G$ is non-abelian,    
then we need to consider the list of groups of order $p^2q$ used in the proof of Theorem~\ref{order prime graph of subgroups of groups 6}. 

\begin{enumerate}[\normalfont (a)]
\item From the structure of $\mathcal P(G_1)$, it follows that $\mathcal P(G_1)$ is planar,  $K_{2,2}$-free and it contains $K_{1,4}$. 

\item Note that $G_2$ has unique subgroup of order $q$.  $\mathcal P(G_2)$ is planar, since $\mathcal P(G_2)$ is bipartite with one partition contains all the subgroups whose order has $p$ as a divisor, and another partition contains the subgroup of order $q$. So $\mathcal P(G_2)$ contains $K_{1,4}$ as a subgraph with bipartition $X:=\{H_1, H_2, H_3, H_4\}$ and $Y:=\{H_5\}$, where $H_1$, $H_2$, $H_3$, $H_4$, $H_5$ are subgroups of $G$ of order $p$, $p$, $p$, $p^2$, $q$ respectively. Clearly $\mathcal P(G_2)$ is $K_{2,2}$-free.

\item  From the structure of $\mathcal P(G_3)$, it follows that $\mathcal P(G_3)$ is planar,  $K_{2,2}$-free and it contains $K_{1,4}$.

\item From the structure of $\mathcal P(G_4)$, it follows that $\mathcal P(G_4)$ is planar,   and it 
 contains $K_{2,3}$, $K_{1,4}$ as a subgraph. 

\item  $G_{5(t)}$ has at least three 
subgroups of order $p$, and at least three subgroups of order $q$. So $\mathcal P(G_{5(t)})$ contains $K_{3,3}$ as a subgraph with bipartition $X:=\{H_1, H_2, H_3\}$ and $Y:=\{H_4, H_5, H_6\}$, where $H_i$, $i=1$, 2, 3 and $H_j$, $j=4$, 5, 6 are subgroups of $G$ of order $p$, $p$, $p$, $q$, $q$, $q$ respectively. Also $\mathcal P(G_{5(t)})$ contains $K_{1,4}$ as a subgraph with bipartition $X:=\{H_1, H_2, H_3, H_7\}$ and $Y:=\{H_4\}$, where $H_7$ is a subgroup of $G$ of order $p^2$.

\item From the structure of $\mathcal P(G_6)$, it follows that $\mathcal P(G_6)$ contains $K_{3,3}$ as a subgraph, so it is non-planar; also it contains $K_{1,4}$ as a subgraph.
\item From the structure of $\mathcal P(D_{12})$, it follows that $\mathcal P(D_{12})$ is planar, $K_{2,2}$-free and it contains $K_{1,4}$ as a subgraph. Also from the structure of $\mathcal P(A_4)$, it follows that $\mathcal P(A_4)$ contains $K_{3,3}$ and $K_{1,4}$ as a subgraph.
\end{enumerate}

\noindent \textbf{Case 3:} Let $\alpha=\beta=1$. Then $G$ has a unique subgroup of order $q$, let it be $H$; it has $q$ subgroups of order $p$; also these are the only proper subgroups of $G$. It follows that $\mathcal P(G)\cong K_{1,q}$, which is planar, acyclic and $K_{2,2}$-free; it is $K_{1,4}$-free if and only if $q=3$; it contains $K_{1,3}$. 
 

The proof follows by combining all the above cases together.
\end{proof}

\begin{pro}\label{planarity of prime graph 8}
If $G$ is a $p$-group, then $\mathcal P(G)$ is planar, $K_{2,3}$-free, $K_{2,2}$-free, $K_{1,4}$-free, $K_{1,3}$-free, $K_{1,2}$-free and acyclic.
\end{pro}
\begin{proof}
If $G$ is a $p$-group, then order of every subgroup of $G$ is power of $p$ and so no two subgroups are adjacent in $\mathcal P(G)$. Thus $\mathcal P(G)$ is totally disconnected and this completes the proof.
\end{proof}

Putting together all the Propositions proved so for in this section, we obtain the  main Theorem \ref{order prime graph of subgroups 9}.

%

\end{document}